\documentclass[12pt]{article}
\usepackage{xypic}
\usepackage{epsfig}
\usepackage{tikz}
\usepackage{graphicx}
\usepackage{amsthm}
\usepackage{amssymb}
\usepackage{caption}
\usepackage{subcaption}
\usepackage{color}
\usepackage{hyperref}
\usepackage{amsmath,mathrsfs,amsfonts,verbatim,enumitem,leftidx}
\usepackage{float}

\newtheorem{defn}{Definition}[section]
\newtheorem{prop}{Proposition}[section]
\newtheorem{example}{Example}[section]
\newtheorem{thm}{Theorem}[section]
\newtheorem{lem}{Lemma}[section]
\newtheorem{rem}{\bf Remark}[section]

\numberwithin{equation}{section}
\newtheorem{claim}{Claim}

\begin{document}

\title{ Mean dimension theory in symbolic dynamics   for  finitely generated  amenable groups 
 \footnotetext {* Corresponding author}
  }
\author{Yunping Wang$^{1}$, Ercai Chen$^{2}$, Xiaoyao Zhou*$^{2}$\\
  \small 1 School of Science, Ningbo Unversity of Technology,\\
  \small  Ningbo 315211, Zhejiang, P.R.China\\
  \small 2 School of Mathematical Sciences and Institute of Mathematics, Nanjing Normal University,\\ 
   \small   Nanjing 210046, Jiangsu, P.R.China\\ 
      \small    e-mail:  yunpingwangj@126.com,  ecchen@njnu.edu.cn,\\ \small  zhouxiaoyaodeyouxian@126.com
}
\date{}
\maketitle

\begin{center}
 \begin{minipage}{120mm}
{\small {\bf Abstract.} In this paper,  we mainly elucidate a close  relationship   between the topological entropy  and   mean dimension theory for actions of polynomial growth groups.  We show that  metric mean dimension and mean Hausdorff  dimension of subshifts  with respect to  the lower rank subgroup are  equal to its topological entropy  multiplied by the   growth rate   of the subgroup.  Meanwhile, we also prove the above result holds for  the rate distortion dimension of subshifts with respect to  the lower rank subgroup and  measure entropy. Furthermore, 
some relevant examples are indicated.
   }
\end{minipage}
 \end{center}

\vskip0.5cm

\section{Introduction}
Let $(X, G)$ be a $G$-action topological dynamical system, where $X$ is a compact Hausdorff space and $G$ a topological group.  Throughout this paper, $G$ is a finitely generated amenable groups. 
 An important dynamical quantity of a shift is its entropy, which roughly measures the exponential growth rate of its projections on  finite sets. For the case $G=\mathbb{N}$, we consider the one-sided infinite product $A^{\mathbb{N}}$ with  the shift map $\sigma: A^{\mathbb{N}} \rightarrow A^{\mathbb{N}}$ defined by
 $$\sigma((x_{n})_{n\in \mathbb{N}})=(x_{n+1})_{n\in \mathbb{N}}.$$
Define a metric which is compatible with the product topology on $A^{\mathbb{N}}$ as follows: for every $x=(x_{n})_{n\in \mathbb{N}}, y=(y_{n})_{n\in \mathbb{N}}\in A^{\mathbb{N}}$,
\begin{align*}
d(x, y)=2^{-\min \left\lbrace n| x_{n}\neq y_{n} \right\rbrace }.
\end{align*}
 Let $\mathcal{X}$ be a closed invariant subset of $A^{\mathbb{N}}$. Furstenberg  proved the following relationship among  entropy, Hausdorff  and Minkowski dimensions of $\mathcal{X}$ with respect to $d$ \cite[Proposition III.1]{Fur67}:
 \begin{align*}
 {\rm dim}_{H}(\mathcal{X}, d)= {\rm dim}_{M}(\mathcal{X}, d)= {h_{top}(\mathcal{X}, \sigma)},
 \end{align*} 
 where $h_{top}(\mathcal{X}, \sigma)$ is the topological entropy of $(\mathcal{X}, \sigma)$.
 Simpson \cite{Sim} generalized the above results  to $\mathbb{Z}^{k}$  action and more general result for amenable group action appears in $\cite{Dou 17}$. For more relevant studies one may refer to   \cite{Chen, Peres}.
 
Mean  dimension is a conjugacy invariant of topological dynamical systems which was introduced by  Gromov \cite{GRO}. This is a dynamical version of topological dimension and it counts how many parameters per iterate  we need to describe an orbit in the dynamical systems. This invariant  has  several applications which cannot be touched within  the framework of topological entropy, see \cite{Tsu18, LL18, MT19}. In particularly, it  has  many applications to embedding problem whether a dynamical system can be embedded into another or not, see for instance \cite{LWE, LT,  GLT16, GYM,  GT20}.   

It is well known  that the concepts of entropy and dimension are closely connected. So it is natural to except we can approach to mean dimension from the entropy theory viewpoint. The first attempt of such an approach was made by Lindenstrauss and Weiss \cite{LWE}. They  introduced the notion of \emph{metric mean dimension}, which is a dynamical analogue of Minkowski dimension  \cite{LWE}, and they proved that metric mean dimension is  an upper bound of the  mean dimension. It allowed them to establish the relationship between the mean dimension and the topological
entropy of dynamical systems. Namely,
  each system with finite topological entropy has zero metric mean dimension and zero mean dimension.  Lindenstrauss and  Tsukamoto in \cite{LT18} established a variational principle between the metric mean dimension and the rate distortion function   under a mild condition on the metric $d$ (called  tame growth of covering numbers, for this definition see \cite{LT18}). Inspired by the classic variational principle of entropy,   they \cite{LT19} also considered  a measure-theoretic notion of mean dimension-\emph{rate distortion dimension}, which was first introduced by Kawabata and Dembo in \cite{KD94} and proved a corresponding variational principle for mean dimension. 
  In order to link  the  measure theoretic aspect of   mean dimension theory,  they  introduced  the  \emph{mean Hausdorff dimension} in  \cite{LT19}, which is a dynamical analogue of Hausdorff dimension. 

 Recently,  Shinoda and  Tsukamoto \cite{MAS} generalized Furstenberg's  result in \cite{Fur67} to $\mathbb{Z}^{2}$  action  which involves metric mean  dimension, mean Hausdorff dimension and rate distortion dimension.  
 In this paper, by adopting the method of  \cite{MAS} and \cite{Dou 17}, we  are going to  prove  the relationship between mean dimension quantities (metric mean dimension, mean Hausdorff dimension and rate distortion dimension) and entropy,  which  generalize the result of \cite{MAS} to actions of polynomial growth groups. 
 The main difficulty in carrying out this generalization is that we need  a Vitali type covering lemma. To this aim we apply a more general covering lemma developed by Lindentrauss \cite{LIN}. 
 
 The paper is organized as follows. In section \ref{s1}, we review basic definitions of finitely generated amenable groups and  mean dimension theory. Meanwhile, we state our main results. In section \ref{s2}, we introduce covering lemma and give the proof of Theorem \ref{main}. In section \ref{s3}, we present the  notion of  rate distortion dimension and  prove the Theorem \ref{11} by following Shinoda and Tsukamoto's technical line. In section \ref{s4}, we give some examples to illustrate our main theorem.

\section{Preliminaries}\label{s1}
In this section, we review some of the standard  concepts and results  on  finitely generate amenable groups,  metric mean dimension and mean Hausdorff dimension. Finally, we state our main results.
\subsection{Finitely generated amenable groups}
Let $G$ be an infinite discrete countable group. Let $F(G)$ denote the set of all finite non-empty subset of $G$. For $K$, $F\in F(G)$, let $KF\in F(G)$, let $KF=\left\lbrace st: s\in K, t\in F\right\rbrace $ and $KF\Delta F=(KF\setminus F) \cup (F\setminus KF)$. A group $G$ is called \emph{amenable} if for each $K \in F(G)$ and $\delta>0$, there exists $F\in F(G)$ such that $|KF\Delta F|< \delta |F|$, where $|\cdot|$ is the counting measure.

Let $K\in F(G)$ and $\delta>0$. A finite subset $A\in F(G)$ is called \emph{$(K, \delta)$-invariant}  if 
$$\dfrac{|B(A, K)|}{|A|}<\delta,$$
where $B(A ,K)$, the\emph{$K$-boundary}  of $A$, is defined by 
\begin{align*}
B(A, K)=\left\lbrace g\in G: Kg \cap A\neq \emptyset ~\text{and} ~ Kg \cap (G\setminus A)\neq \emptyset\right\rbrace .
\end{align*}
Another equivalent condition for the sequence of finite subset $\left\lbrace F_{n}\right\rbrace $ of $G$ to be a F$\phi$lner sequence is that $\left\lbrace F_{n}\right\rbrace $ becomes more and more invariant, i.e., for each $K\in F(G)$ and $\delta>0$, $F_{n}$ is $(K, \delta)$-invariant when $n$ is large enough. 
A group $G$ is  amenable group if and only if $G$ admits a F$\phi$lner sequence $\left\lbrace F_{n} \right\rbrace $.  For more details and properties of the amenable group, one is referred to \cite{OW} \cite{Coo15}.
Let $\epsilon\in (0,1)$. $A_{1}, A_{2}, \cdots, A_{k} \in F(G)$ are said to be \emph{$\epsilon$-disjoint }if there exist mutually disjoint $A_{i}'\subset A_{i}$ such that $|A_{i}'|\geq(1-\epsilon)|A_{i}|$ for $1\leq i \leq k$. 
Recall that a F$\phi$lner sequence $\left\lbrace F_{n}\right\rbrace $ in $G$ is said to be \emph {tempered} if there exists a constant $C>0$ which is independent of $n$ such that 
\begin{align}
|\bigcup\limits_{k<n} F_{k}^{-1} F_{n}|\leq C|F_{n}|, ~\text{for any} ~n.
\end{align}

Let $G$ be a finitely generated amenable group  with  a symmetric generating set $S$. Recall that a generating set is called \emph{symmetric} if together with any $s\in S$ it contains $s^{-1}$. 
The $S$-\emph{word-length} $\ell_{S}(g)$ of  an element $g\in G$ is the minimal integer $n\geq 0$ such that $g$ can be expressed as a product of $n$ elements in $S$, that is, 
\begin{align*}
\ell_{S}(g)=\min \left\lbrace n\geq 0: g=s_{1}\cdots s_{n}: s_{i}\in S, 1\leq i \leq n \right\rbrace. 
\end{align*} 
It immediately follows from the definition that for $g\in G$ one has
$\ell_{S}(g)=0~ \text{if an only if} ~g= 1_{G}.$
Define the metric $d_{S}$ on $G$ :
$d_{S}(g, h)=\ell_{S}(g^{-1}h).$
It is obvious that the metric $d_{S}$ is invariant by left multiplication. 

For $g\in G$ and $n\in \mathbb{N}$, we denote by
\begin{align*}
B_{S}^{G}(g, n)=\left\lbrace h\in G: d_{S}(g, h)\leq n \right\rbrace,
\end{align*}
the ball of radius $n$ in $G$ centered at the element $g\in G$. 
When $g=1_{G}$ we have $B_{S}^{G}(1_{G}, n)=\left\lbrace h\in G: \ell_{S}(h)\leq n \right\rbrace $ and we simply write $B_{S}^{G}(n)$ instead of $B_{S}^{G}(1_{G}, n)$. Also, when there is no ambiguity on the group $G$, we omit the subscript $G$ and we simply write  $B_{S}(g, n)$ and $B_{S}(n)$ instead of $B_{S}^{G}(g, n)$ and $B_{S}^{G}(n)$.
The \emph{growth function} of $G$ relative to $S$ is a function $\gamma_{S}: \mathbb{N} \rightarrow \mathbb{N}$ defined by
$\gamma_{S}(n)=|B_{S}(n)|=|\left\lbrace g \in G: \ell_{S}(g)\leq n \right\rbrace |.$

\begin{rem}\cite{CC}
	Let $G_{1}$ and $G_{2}$ be two finitely generated  amenable groups. Then the direct product $G_{1}\times G_{2}$ is also a finitely generated amenable group.
\end{rem}
\begin{defn}
Let $G$ be a finitely generated group of polynomial growth  with a symmetric generating set $S$  if there exists constants $d, A, B>0$ such that $$An^{d}\leq \gamma_{S}(n) \leq B n^{d}$$  for all $n\in \mathbb{N} .$ We denote by ${\rm deg}(G)=d$ the degree of the polynomial growth of $G$.	
\end{defn}
 In this paper, we consider   a finitely generated group of polynomial growth. Polynomial growth groups are amenable. We will review the definitions of metric mean dimension \cite{Dou}  and introduce the mean Hausdorff dimension of amenable group actions  in subsection \ref{sub}.


 \subsection{Metric mean dimension and mean Hausdorff dimension}\label{sub}
 Let $G$ be a  countable discrete  amenable group. Let $(\mathcal{X}, G)$ be  a $G$-system with $d$. For $\epsilon>0$, we define $\#(\mathcal{X}, d, \epsilon)$ as the minimum natural number $n$ such that $\mathcal{X}$ can be covered by open sets $U_{1}, \cdots, U_{n}$ with ${\rm diam}(U_{i})<\epsilon$ for $1\leq i \leq n$. For $F\in F(G)$, define metric $d_{F}$ on $\mathcal{X}$ by 
 $$d_{F}(x, y)=\max\limits_{g \in F} d(gx, gy).$$ 
 For $s\geq 0$ and $\epsilon>0$, we define $H_{\epsilon}^{s}(\mathcal{X}, d)$ as 
 \begin{align}
 \inf\left\lbrace \sum\limits_{n=1}^{\infty}({\rm diam} E_{n})^{s}| \mathcal{X}=\bigcup\limits_{n} E_{n} ~\text{with}~ {\rm diam} E_{n}<\epsilon ~\text{for all}~ n\geq 1\right\rbrace .
 \end{align}
 We set $${\rm dim}_{H}(\mathcal{X}, d, \epsilon)=\sup \left\lbrace s\geq 0 | H_{\epsilon}^{s}(\mathcal{X}, d)\geq 1  \right\rbrace. $$ 
 The Hausdorff dimension ${\rm dim}_{H}(\mathcal{X}, d)$ is given by 
 $${\rm dim}_{H}(\mathcal{X},d )=\lim\limits_{\epsilon \rightarrow 0} {\rm dim}_{H}(\mathcal{X}, d, \epsilon).$$
 We define the upper and lower mean Hausdorff dimension. 
  Let $\left\lbrace F_{n}\right\rbrace $ be a F$\phi$lner sequence in $G$, we can define
 \begin{align*}
 \overline{\rm mdim}_{H}(\mathcal{X}, \left\lbrace F_{n}\right\rbrace , d)=\lim\limits_{\epsilon\rightarrow 0} \left( \limsup\limits_{n\rightarrow \infty} \dfrac{{\rm dim}_{H}(\mathcal{X}, d_{F_{n}}, \epsilon)}{|F_{n}|}\right),\\ 
 \underline{\rm mdim}_{H}(\mathcal{X}, \left\lbrace F_{n}\right\rbrace, d)=\lim\limits_{\epsilon\rightarrow 0} \left( \liminf\limits_{n\rightarrow \infty} \dfrac{{\rm dim}_{H}(\mathcal{X}, d_{F_{n}}, \epsilon)}{|F_{n}|}\right) .
 \end{align*}
 When these two quantities are equal to each other, we denote the common value by ${\rm mdim}_{H}(\mathcal{X},\left\lbrace F_{n}\right\rbrace, d)$.
 For any $\epsilon>0$, we define
 $$ S(\mathcal{X}, G, d, \epsilon)=\lim\limits_{n \rightarrow \infty} \dfrac{1}{|F_{n}|}\log \# (\mathcal{X}, d_{F_{n}}, \epsilon).$$
 The limit always exists and does not depend on the choice of the F$\phi$lner sequence $\left\lbrace F_{n} \right\rbrace $. The upper and lower metric mean dimension is then defined by 
 $$\overline{\rm mdim}_{M}(\mathcal{X}, G, d)=\limsup\limits_{\epsilon\rightarrow 0}\dfrac{S(\mathcal{X}, G, d, \epsilon)}{|\log \epsilon|},$$
 $$\underline{\rm mdim}_{M}(\mathcal{X}, G, d)=\liminf\limits_{\epsilon\rightarrow 0}\dfrac{S(\mathcal{X}, G, d, \epsilon)}{|\log \epsilon|}.$$
 When the upper and lower limits coincide, we denote the common value by ${\rm midm}_{M}(\mathcal{X}, G, d)$. 
 
 The following result is the dynamical analogue of the fact that Minkowski dimension no less than Hausdorff dimension. 
 \begin{prop}
 	Let $\left\lbrace F_{n}\right\rbrace $ be  a F$\phi$lner sequence, then
 	\begin{align*}
 	\overline{{\rm mdim}}_{H}(\mathcal{X},  \left\lbrace F_{n}\right\rbrace , d ) \leq \underline{{\rm mdim}}_{M}(\mathcal{X},  G, d ).
 	\end{align*}
 \end{prop}
\begin{proof}
	Let $\left\lbrace F_{n}\right\rbrace $ be a the F$\phi$lner sequence in $G$. For $n\geq 0$, $\epsilon>0$, choose an open cover $\mathcal{X}=U_{1}\cup \cdots \cup U_{m}$ with ${\rm diam}(U_{i}, d_{F_{n}})\leq \epsilon$ and $m=\#(\mathcal{X}, d_{F_{n}}, \epsilon)$. 
	We have $$H_{\epsilon}^{s}(\mathcal{X}, d_{F_{n}})\leq m \epsilon^{s}.$$
	
	If $s> \log m/ \log (1/ \epsilon)$, then 
	$H^{s}(\mathcal{X}, d_{F_{n}})<1$. This shows $${\rm dim}_{H}(\mathcal{X}, d_{F_{n}}, \epsilon)
	\leq \dfrac{\log \#(\mathcal{X}, d_{F_{n}}, \epsilon)}{\log (1/\epsilon)}.$$
	Divide this by $|F_{n}|$ and take limits with respect to $n$ and then $\epsilon$. It follows that $\overline{\rm mdim}_{H}(\mathcal{X}, \left\lbrace F_{n}\right\rbrace, d)\leq \underline{\rm midm}_{M}(\mathcal{X}, G, d).$
\end{proof}
\subsection{Statement of the main results}
Now we state the main theorems. Let $G_{1}$ and $G_{2}$ be  finitely generated groups of polynomial growth. Then direct product  $G= G_{1} \times G_{2}$ is also finitely generated of polynomial growth. Let $S_{1}$ and $S_{2}$ be finite symmetric generating subsets of $G_{1}$ and $G_{2}$. Then the set
$$
S=(S_{1}\times \left\lbrace 1_{G_{1}}\right\rbrace ) \cup (\left\lbrace 1_{G_{2}}\right\rbrace \times S_{2})
$$
is a finite symmetric generating subset of $G$. We denote by ${\rm deg}(G_{1})$ and ${\rm deg}(G_{2})$  the degrees of the polynoimal growth of $G_{1}$ and $G_{2}$, respectively (e.g., ${\rm deg}(\mathbb{Z}^{k})=k$).
Set $S=\left\lbrace s_{1}, \cdots s_{m}\right\rbrace $. Next we defines a order in $S$ which formalize through the following construction: given $s_{i}, s_{j} \in S$ we say that   $s_{i}<s_{j}$ if $i<j$. 
Hence $s_{1}<s_{2}<\cdots<s_{m}$. 
Hence we can consider the  order in  $G$.  
For $g, g' \in G$, we call  $g< g'$ if $\ell(g)<\ell(g')$. If $\ell(g)=\ell(g')=n$, then there exist $s_{1}, \cdots, s_{n}$ and $s_{1}', \cdots, s_{n}'$ such that 
\begin{align*}
g=s_{1}\cdots s_{n},~g'=s_{1}'\cdots s_{n}'.
\end{align*} 
Take $k=\min \left\lbrace i: s_{i}\neq s_{i}' \right\rbrace $.  When $s_{k}<s_{k}'$, we denote by $g<g'$, otherwise, $g>g'$. Then we can arrange the elements in the group. 
Let $G=(g_{n})_{n=0}^{\infty}$ be an enumeration of $G$ according to the order such that  $\ell(1_{G})=\ell_{S}(g_{0})\leq \ell_{S}(g_{1}))\leq \ell_{S}(g_{2})\cdots$.

We can define  a metric $d$ on $A^{G}$ by the  following: 
\begin{align}\label{metric}
d(x, y)=2^{-\min\left\lbrace |g_{n}|_{\infty} \big| x_{g_{n}}\neq y_{g_{n}}\right\rbrace },
\end{align}
where $|g_{n}|_{\infty}=\max \left\lbrace \ell_{S_{1}}(g_{n,1}), \ell_{S_{2}}(g_{n,2}) \right\rbrace $ and $g=(g_{n,1}, g_{n, 2})$. A closed $G$-invariant subset  $\mathcal{X}$ of $A^{G}$ is called a \emph {subshift} of $A^{G}$.
\begin{thm}\label{main}
	Let $\mathcal{X}\subset A^{G}$ be a subshift. 
	Suppose that ${\rm deg}(G_{2})=1$. Then
	\begin{itemize}
		\item[(1).] $\overline{\rm mdim}_{M}(\mathcal{X}, {G_{1}}, d)\leq c_{1}\cdot h_{top}(\mathcal{X}, {G}),$
		where $c_{1}=\limsup\limits_{n\rightarrow \infty}\frac{|B_{S_{2}}(n)|}{n}$.
		\item[(2).]$	\underline{{\rm mdim}}_{H}(\mathcal{X}, \left\lbrace B_{S_{1}}(n) \right\rbrace , d)\geq c_{2}\cdot h_{top}(\mathcal{X}, G) ,$ where $c_{2}=\liminf\limits_{n\rightarrow \infty}\frac{|B_{S_{2}}(n)|}{n}.$
	\end{itemize}
	In particular, if $c_{1}=c_{2}=c$, we have 
	\begin{align*}
	{\rm mdim}_{H}(\mathcal{X}, \left\lbrace B_{S_{1}}(n) \right\rbrace , d)={\rm mdim}_{M}(\mathcal{X}, {G_{1}}, d)= c\cdot h_{top}(\mathcal{X}, {G}).
	\end{align*}
\end{thm}
\begin{thm}\label{11}
	If  $\mu$ is a Borel probability measure on $\mathcal{X}$ invariant under both $\sigma_{1}$ and $\sigma_{2}$ and $c_{1}=c_{2}=c$, then
	\begin{align*}
	{\rm rdim}(\mathcal{X}, \sigma_{1}, \left\lbrace B_{S_{1}}(n) \right\rbrace,  d, \mu)=c\cdot h_{\mu}(\mathcal{X}, G).
	\end{align*}
\end{thm}
\begin{rem}
	\begin{itemize}
		\item[(1)] It is not clear whether  the value of $c$  exists for general amenable groups. Hence we only consider the polynomial growth groups in this paper.
		\item[(2)] 
		The reason for imposing the condition  ${\rm deg}(G_{2})=1$ is that if ${\rm deg}(G_{2})>1$, we don't know whether   the value of $c$ exists.   
	\end{itemize}
\end{rem}

\section{Proof of Theorem \ref{main}} \label{s2}
For a finite alphabet $A$ and  a finitely generated group of polynomial growth  $G$,  the full $G$-shift over $A$ is the set $A^{G}$, which is viewed as a compact topological space with the discrete product topology.
Consider the shift action on  the product space $A^{G}$:
$$g' (x_{g})_{g\in G}=(x_{gg'})_{g\in G}, \text{for all}~ g'\in G ~\text{and
}~ (x_{g})_{g\in G} \in A^{G}.$$
We define the shifts $\sigma_{1}$ and $\sigma_{2}$ on $A^{G}$ by
$(\sigma_{1, g}x)_{(g_{1}, g_{2})}=x_{(g_{1}g, g_{2})}$
and
$(\sigma_{2, h}x)_{(g_{1}, g_{2})}=x_{(g_{1}, g_{2}h)}$ 
for all $g\in G_{1}$, $h\in G_{2}$, $(g_{1}, g_{2})\in  G$. Let $\mathcal{X}$ be a subshift.
For $E\subset G$, let $\pi_{E}: \mathcal{X}\rightarrow A^{E}$ denote the canonical projection map, that is, the map defined by $\pi_{E}(x):=x|_{E}$ for all $x\in \mathcal{X}$, where $x|_{E}$ denote the restriction of  $x: G \rightarrow A$  to $E \subset G$. For each finite set $F\subset G$ and  $\omega\in A^{F}$, a subset $C\subset \mathcal{X}$ is called  a \emph{cylinder} over $F$ if there exists $x\in \mathcal{X}$ such that  $C$ is equal to the set of all $x\in \mathcal{X}$ with $\pi_{F}(x)=\pi_{F}(y)$.
The proof falls naturally into two steps.

\noindent\textbf{Step 1:}  
$\overline{\rm mdim}_{M}(\mathcal{X}, G_{1}, d)\leq c_{1} \cdot h_{top}(\mathcal{X}, G).$
\begin{proof}
	For $\epsilon>0$, choose $M>0$ such that $2^{-M}<\epsilon<2^{-M+1}$. For each a natural number $N>0$, then
\begin{align*}
\#(\mathcal{X}, d_{B_{S_{1}}(N)}^{\sigma_{{1}}}, \epsilon)\leq |\pi_{(B_{S_{1}}(M)B_{S_{1}}(N)\times B_{S_{2}}(M))}(\mathcal{X})|,
\end{align*}
where  $d_{B_{S_{1}}(N)}^{\sigma_{1}}(x, y)=\max_{g\in B_{S_{1}}(N)} d(\sigma_{1, g} x, \sigma_{1, g} y) $.
Since $M-1\leq \log (1/\epsilon)\leq M$, 
	\begin{align*}
	&\overline{\rm mdim}_{M}(\mathcal{X}, G_{1}, d)\\=&\limsup\limits_{\epsilon \rightarrow 0} \left( \lim\limits_{N\rightarrow \infty} \dfrac{\log \# (\mathcal{X}, d_{B_{S_{1}}(N)}^{\sigma_{{1}}}, \epsilon) }{|B_{S_{1}}(N)|\log (1/\epsilon)}\right)\\\leq& \lim\limits_{M\rightarrow \infty} \left(\lim\limits_{N\rightarrow \infty} \dfrac{|\pi_{(B_{S_{1}}({M})B_{S_{1}}(N)\times B_{S_{2}}({M}))}(\mathcal{X})| }{|B_{S_{1}}(N)|({M-1})} \right) \\\leq& \lim\limits_{M\rightarrow \infty} \left(\lim\limits_{N\rightarrow \infty}\dfrac{|\pi_{(B_{S_{1}}({M})B_{S_{1}}(N)\times B_{S_{2}}({M}))}(\mathcal{X})| }{|B_{S_{1}}({M})B_{S_{1}}(N)||B_{S_{2}}({M})|} \times \dfrac{|B_{S_{1}}({M})B_{S_{1}}(N)||B_{S_{2}}({M})|}{|B_{S_{1}}(N)|({M-1})} \right) 
	\end{align*}
	Since $\left\lbrace B_{S_{1}}(N) \right\rbrace $ is a F$\phi$lner sequence,
then $\lim\limits_{N\rightarrow \infty} \dfrac{|B_{S_{1}}({M})B_{S_{1}}(N)|}{|B_{S_{1}}(N)|}=1$. 
Note that $\limsup\limits_{M\rightarrow \infty} \dfrac{|B_{S_{2}}(M)|}{{M}}=c_{1}$. We can get the desired result. 
\end{proof}	
We next give the 	 following covering lemma which  was proved  by  Linedentrauss \cite{LIN}. This lemma is crucial in the proof of  step 2 of Theorem \ref{main}. 
\begin{lem}\label{cc}\cite{LIN}
	For any $\delta\in (0, 1/100)$, $C>0$ and finite $D\subset G$, let $M\in \mathbb{N}$ be sufficiently large (depending only on $\delta$, $C$ and $D$). Let $F_{i, j}$ be an array of a finite subsets of $G$ where $i=1, \cdots, M$ and $j=1, \cdots, \ell_{i}$, such that 
	\begin{itemize}
		\item For every $i$, $\overline{F}_{i, *}=\left\lbrace F_{i, j}\right\rbrace_{j=1}^{\ell_{i}}$ satisfies 
		\begin{align*}
		|\bigcup\limits_{k'<k}F_{i, k'}^{-1} F_{i, k}|\leq C|F_{i, k}|, \text{for}~ k=2, \cdots, \ell_{i}.
		\end{align*}
		Denote $F_{i, *}=\bigcup \overline{F}_{i, *}$.
		\item The finite set sequence $F_{i, *}$ satisfy that for every $1\leq i \leq M$ and  every $1\leq k \leq \ell_{i}$,
		\begin{align*}
		|\bigcup\limits_{i'<i}DF_{i', *}^{-1} F_{i,k}|\leq (1+\delta)|F_{i, k}|.
		\end{align*}
	\end{itemize}	
Assume that  $A_{i, j}$ is another array of finite subset of $G$ with $F_{i, j}A_{i,j}\subset F$ for some finite subset $F$ of $G$. Let $A_{i, *}=\cup_{j}A_{i, j}$ and 
\begin{align*}
\alpha=\dfrac{\min\limits_{1\le i \leq M}|DA_{i, *}|}{|F|}.
\end{align*}
Then the collection of subsets of $F$, 
\begin{align*}
\overline{F}=\left\lbrace F_{i,j} a: 1\leq i \leq M, 1 \leq j \leq \ell_{i} ~ \text{and}~ a \in A_{i, j}\right\rbrace 
\end{align*}
has a subfamily $\mathcal{F}$ that is $10 \delta^{1/4}$-disjoint such that 
\begin{align*}
|\cup \mathcal{F}|\geq (\alpha- \delta^{1/4})|F|.
\end{align*}
\end{lem}

\noindent	\textbf {Step 2: } $\underline{{\rm mdim}}_{H}(\mathcal{X}, \left\lbrace B_{S_{1}}(N) \right\rbrace, d)\geq c_{2} \cdot h_{top}(\mathcal{X}, G)$.
\begin{proof}
	Set 
$s=c_{2}\cdot h_{top}(\mathcal{X}, G),~h=h_{top}(\mathcal{X}, G).$
We suppose  $\underline{\rm mdim}_{H}(\mathcal{X}, \left\lbrace B_{S_{1}}(n)\right\rbrace , d)<s$. We would like to get a contradiction. Take $\epsilon$ such that $s\epsilon-\frac{\epsilon^{2}}{2}<1$ and 
\begin{align*}
\underline{\rm mdim}_{H}(\mathcal{X}, \left\lbrace B_{S_{1}}(n)\right\rbrace , d)<s-(h+2)\epsilon.
\end{align*}
Let $D=\left\lbrace e_{G} \right\rbrace \subset G$ and  $C>0$ be the constant in the tempered condition for the F$\phi$lner sequence $\left\lbrace B_{S}(n)\right\rbrace $. We choose  $0< \delta < \min \left\lbrace 1/100,\epsilon \right\rbrace $ small enough and  a natural number $M$   satisfying  the following  conditions:
\begin{align}\label{da}
H(\delta)+\delta\log M <\epsilon^{3}/4c_{2},~|A|^{\delta}<2^{\epsilon^{3}/4c_{2}},~1/(1-10\delta^{1/4})<1+\epsilon^{2}.
\end{align}
Here $H(\delta)=-\delta \log \delta-(1-\delta)\log (1-\delta)$. (Recall that the base of the logarithm is two.) Take $M\approx \log C/\delta^{2}$ to satisfy  the requirement of Lemma \ref{cc} corresponding to $\delta, D$ and $C$. Then we can choose a sufficiently small $\delta$ satisfying the second and third conditions.

We choose a natural number $r_{0}$ such that 
\begin{align}\label{1}
r_{0} >\dfrac{1}{\delta (1- 10\delta^{1/4})}.
\end{align}
and 
\begin{align}\label{cd}
(s-(h+2)\epsilon)r<(s-(h+1)\epsilon)(r-1),~~|B_{S_{2}}(r)|>r (c_{2}-\epsilon)
\end{align}
for every $r\geq r_{0}$.

From $\underline{\rm mdim}_{H}(\mathcal{X}, \left\lbrace B_{S_{1}}(N) \right\rbrace,d )<s-(h+2)\epsilon$, for each $i=1, 2, \cdots M$,  we can find $N_{i}>0$  satisfying 
\begin{align*}
\dfrac{1}{|B_{S_{1}}(N_{i})|} {\rm dim}_{H}(\mathcal{X}, d_{B_{S_{1}}(N_{i})}^{\sigma_{1}}, 2^{-r_{0}})<s-(h+2)\epsilon.
\end{align*}
Also, we let the sequence  $\left\lbrace N_{i}\right\rbrace $ be increasing.
This implies that there exists a covering $\mathcal{X}=\cup_{j=1}^{l_{i}} E_{j}^{i} $ satisfying 
\begin{align*}
{\rm diam}(E_{j}^{i}, d_{B_{S_{1}}(N_{i})}^{\sigma_{1}})<2^{-r_{0}}(1\leq j \leq l_{i}), \sum_{j=1}^{l_{i}} ({\rm diam}(E_{j}^{i} ,~ d_{B_{S_{1}}(N_{i})}^{\sigma_{1}}))^{(s-(h+2)\epsilon)|B_{S_{1}}(N_{i})|}<1.
\end{align*}
Set $2^{-r_{j}^{i}}={\rm diam}(E_{j}^{i}, d_{A_{N_{i}}}^{\sigma_{1}}).$ Then $r_{j}^{i}$ is a natural number with $r_{j}^{i}>r_{0}$. Choose $x_{j}^{i}\in E_{j}^{i}$,
let $C_{j}^{i}=\pi_{F_{n_{i, j}}}^{-1}\left( \pi_{F_{n_{i,j }}}  (x_{j}^{i})\right) $ be  a  cylinder over  $$F_{n_{i, j}}:=\left\lbrace B_{S_{1}}(r_{j}^{i}-1)B_{S_{1}}(N_{i}) \times B_{S_{2}}(r_{j}^{i}-1)\right\rbrace .$$  
Then $E_{j}^{i}\subset C_{j}^{i}$ and $\mathcal{X}=\cup_{j=1}^{\ell_{i}}C_{j}^{i}$.
For each $1\leq i \leq M$,  we can get $\{F_{n_{i, 1}}, \cdots, F_{n_{i, l_{i}}}\}  $. Without loss of generality, let  $r_{1}^{i}<r_{2}^{i}<\cdots<r_{l_{i}}^{i}$.
Then let $\left\lbrace F_{i,1}, \cdots, F_{i, l_{i}} \right\rbrace $ in Lemma \ref{cc} be as 
\begin{align*}
\left\lbrace F_{i, 1}, \cdots, F_{i, l_{i}}\right\rbrace =\{F_{n_{i, 1}}, \cdots, F_{n_{i, l_{i}}}\}  .
\end{align*}
Since $|B_{S_{2}}(r)|>r (c_{2}-\epsilon)$ and $ s-(h+1)\epsilon<(s-\epsilon)\frac{c_{2}-\epsilon}{c_{2}}$, we have
\begin{align*}
|F_{i, j}|=|B_{S_{1}}(r_{j}^{i}-1)B_{S_{1}}(N_{i}) \times B_{S_{2}}(r_{j}^{i}-1)|\geq |B_{S_{1}}(N_{i})||B_{S_{2}}(r_{j}^{i}-1)|,
\end{align*} 
and 
\begin{align*}
\dfrac{1}{c_{2}}(s-\epsilon)|F_{i,j}|&\geq \dfrac{1}{c_{2}}(s-\epsilon)|B_{S_{1}}((N_{i})||B_{S_{2}}(r_{j}^{i}-1)|\\& > (s-\epsilon)|B_{S_{1}}((N_{i})|\left( \frac{(r_{j}^{i}-1)(c_{2}-\epsilon)}{c_{2}}\right)  
\\&>(s-(h+1)\epsilon)|B_{S_{1}}(N_{i})|(r_{j}^{i}-1)\\&> (s-(h+2)\epsilon)|B_{S_{1}}(N_{i})|r_{j}^{i}.~~(\text{by} ~\ref{cd})
\end{align*}     
Hence 
\begin{align*}
2^{-\frac{1}{c_{2}}(s-\epsilon)|F_{i,j}|}<2^{-(s-(h+2)\epsilon)|B_{S_{1}}(N_{i})|r_{j}^{i}}={\rm diam}(E_{j}^{i}, d_{B_{S_{1}}(N_{i})}^{\sigma_{1}})^{(s-(h+2)\epsilon)|B_{S_{1}}(N_{i})|}.
\end{align*}
It follows that
\begin{align}\label{e1}
\sum\limits_{j=1}^{l_{i}}2^{-\frac{1}{c_{2}}(s-\epsilon)|F_{i,j}|}\leq  \sum\limits_{j=1}^{l_{i}} ({\rm diam}(E_{j}^{i} ,~ d_{A_{N_{i}}}^{\sigma_{1}}))^{(s-(h+2)\epsilon)|B_{S_{1}}(N_{i})|}<1.
\end{align}
For any $x\in \mathcal{X}$ and sufficiently large $N$ (independent on $x$), let 
\begin{align*}
A_{i, j}=\left\lbrace a \in B_{S}(N): F_{i, j}a \subset B_{S}(N) ~\text{and}~ \sigma^{a}x\in C_{ j}^{i} \right\rbrace.
\end{align*}
We note here that $A_{i, j}$ depends on $x$. For any $g\in B_{S}(N)\setminus B(B_{S}(N), F_{i, *} )$, we have $F_{i, *}g \subset B_{S}(N)$. Then $F_{i, j}g \subset B_{S}(N)$ for each $1\leq j \leq l_{i}$.  Since $\left\lbrace C_{j}^{i} \right\rbrace_{j=1}^{l_{i}} $ cover $\mathcal{X}$, there exists $C_{j}^{i}$ such that $\sigma^{g}x \in C_{j}^{i}$. This implies $g\in A_{i, j}$ and $B_{S}(N)\setminus B(B_{S}(N), F_{i,*})\subset A_{i, *}$. Let $N$
 be sufficiently large so that $B_{S}(N)$ is $(F_{i,*}, \delta)$-invariant for all $1\leq i \leq M $, $1\leq j \leq l_{i} $, then
\begin{align*}
\alpha=\dfrac{\min_{1\leq i \leq M}|DA_{i, *}|}{|B_{S}(N)|}>1-\delta.
\end{align*}
We note that the array $\left\lbrace F_{i, j} \right\rbrace  $  meet the first requirement in Lemma \ref{cc} because of the tempered condition of $\left\lbrace B_{S}(N)\right\rbrace $. For the second requirement, we need to choose $N_{i}$ large enough compared with $r_{l_{i-1}}^{i-1}$ for every  $2\leq i \leq M$.
 Now we can apply  Lemma \ref{cc} to $$\overline{\mathcal{F}}=\left\lbrace F_{i, j}a: 1\leq i \leq M,  1 \leq j \leq l_{i} ~\text{and}~a\in A_{i, j} \right\rbrace, $$ we can find a subcollection $\mathcal{F}$ that is $10\delta^{1/4}$-disjoint such that
 \begin{align}\label{a}
 |\cup \mathcal{F}|\geq (1-\delta -\delta^{1/4})|B_{S}(N)|.
 \end{align}    
 The element in $\mathcal{F}$ will be denoted by $F_{i, j}a$. 
Denote by $\overline{A}$ the collection of $a$'s such that $F_{i,j}a$ occurs in $\mathcal{F}$. The  cardinality of $\overline{A}$ is no more than the cardinality of the subcollection $\mathcal{F}$.
 Then $|\overline{A}|\leq |\mathcal{F}|$.
 Note that $\mathcal{F}$ is $10\delta ^{1/4}$-disjoint, we have 
 \begin{align*}
 \sum\limits_{F_{i, j}a \in  \mathcal{F}}|F_{i,j}a| \leq \dfrac{1}{1-10\delta^{1/4}}|\cup \mathcal{F}|\leq \dfrac{1}{1-10\delta^{1/4}} |B_{S}(N)|.
 \end{align*}
 Then
 \begin{align*}
 |\overline{A}|\leq |\mathcal{F}| \leq \dfrac{1}{\min{|F_{i, j}|}} \cdot \dfrac{1}{1-10\delta^{1/4}}|B_{S}(N)|\leq \delta |B_{S}(N)|.~~(\text{by} ~(\ref{1}))
 \end{align*}
Set 
\begin{align*}
D(x)=\left\lbrace (a, i, j): a\in A_{i,j}~1\leq i\leq M, ~1\leq j\leq \ell_{i} ~\text{such that}~F_{i,j}a \in \mathcal{F} \right\rbrace. 
\end{align*} 
 By the above argument, if $N>0$ is sufficiently large, we have the following conclusion: 
\begin{itemize}
	\item [(1)] For each $(a, i, j)\in D(x)$, we have $\sigma^{a} x\in  C_{j}^{i}$  and $F_{i,j}a \subset \mathcal{F}$.
	\item[(2)] If $(a, i, j)$ and $(a', i', j')$  are two different element of $D(x)$, then  $F_{i, j, }a$ and $F_{i', j'}a'$ is $10\delta^{1/4}$ disjoint. 
	\item[(3)] $|\bigcup \limits_{(a,i,j)\in D(x)} F_{i, j}a|\geq (1-\delta -\delta^{1/4})|B_{S}(N)|$.
\end{itemize}
 For each $x\in \mathcal{X}$ we define $\underline{D}(x)\subset B_{S}(N)\times [1, M]$ as the set of  $(a, i)$ such that there exists $j\in [1, \ell_{i}]$ with $(a, i, j)\in D(x)$. (Note that the sets $D(x)$ and $\underline{D}(x)$ depend on $N$). For simplicity of notation, we write $D(x)$ and $\underline{D}(x)$ instead of $D_{N}(x)$ and $\underline{D}_{N}(x)$.

\begin{claim}
	If $N$ is sufficiently large then the number of possibilities  of $\underline{D}(x)$ is bounded as follows:
	\begin{align*}
	|\left\lbrace \underline{D}(x)| x \in \mathcal{X} \right\rbrace |<2^{(\epsilon^{3}/4c_{2})|B_{S}(N)|}.
	\end{align*}
\end{claim}

\begin{proof}
	It is well-known that 
	\begin{equation*} \tbinom{n}{k}\leq  2^{n H(k/n)}.
	\end{equation*}
	 Since $|\overline{A}|\leq \delta|B_{S}(N)|$, 
	  then 
	 the number of possibilities of $\underline{D}(x)$ is bounded by 
	 \begin{align*}
	\sum\limits_{k=1}^{\delta|B_{S}(N)|}\tbinom{|B_{S}(N)|}{k}\times M^{\delta|B_{S}(N)|}&\leq |B_{S}(N)|\cdot 2^{|B_{S}(N)| H(\delta)}\times 2^{|B_{S}(N)|\delta\log M}\\&=|B_{S}(N)|\cdot 2^{|B_{S}(N)|(H(\delta)+\delta \log M) }.
	 \end{align*}
	 We assume $H(\delta)+\delta \log M <(\epsilon^{3}/4c_{2})$ in (\ref{da}). Hence, if $N$ is sufficiently large then 
	 \begin{align*}
	 |B_{S}(N)|\cdot 2^{|B_{S}(N)|(H(\delta)+\delta \log M) }<2^{(\epsilon^{3}/4c_{2})|B_{S}(N)|}.
	 \end{align*}
\end{proof}
Take a subset $E\subset  B_{S}(N)\times [1, M]$ such that there exists $x\in \mathcal{X}$ with $\underline{D}(x)=E$. We denote by $\mathcal{X}_{E}$ the set of $x\in \mathcal{X}$ with $\underline{D}(x)=E$. Let $E=\left\lbrace (a_{1}, i_{1}), (a_{2}, i_{2}), \cdots, (a_{k}, i_{k})\right\rbrace $.
\begin{claim}
	\begin{align*}
	|\pi_{B_{S}(N)}(\mathcal{X}_{E})|\cdot 2^{-\frac{1}{c_{2}}(s-\epsilon)(1+\epsilon^{2})|B_{S}(N)|}\leq |A|^{(\delta+\delta^{1/4})|B_{S}(N)|}.
	\end{align*}
\end{claim}
 \begin{proof}
 	For ${\bf j}=(j_{1}, \cdots, j_{k})\in [1, \ell_{i_{1}}]\times \cdots \times [1, \ell_{i_{k}}]$, we denote by $\mathcal{X}_{E, {\bf j}}\subset \mathcal{X}_{E}$ the set of $x\in \mathcal{X}_{E}$ with $D(x)=\left\lbrace (a_{1}, i_{1}, j_{1}), \cdots, (a_{k}, i_{k}, j_{k}) \right\rbrace $. We have $\sigma^{a_{m}}(x)\in C_{j_{m}}^{i_{m}}$ for $x\in \mathcal{X}_{E, {\bf j}}$. Therefore we have
 	\begin{align*}
 	|\pi_{B_{S}(N)}(\mathcal{X}_{E, {\bf j}})|\leq |A|^{(\delta+\delta^{1/4})|B_{S}(N)|}.
 	\end{align*}
 Here the inequality follows from  (\ref{a}). This follows from
 \begin{align*}
 |\pi_{B_{S}(N)}(\mathcal{X}_{E})|\cdot 2^{-\frac{1}{c_{2}}(s-\epsilon)(1+\epsilon^{2})|B_{S}(N)|}&=\sum\limits_{{\bf j}}|\pi_{B_{S}(N)}(\mathcal{X}_{E, \bf j})|\cdot2^{-\frac{1}{c_{2}}(s-\epsilon)(1+\epsilon^{2})|B_{S}(N)|}\\&\leq \sum\limits_{{\bf j}}|A|^{(\delta+\delta^{1/4})|B_{S}(N)|} 2^{-\frac{1}{c_{2}}(s-\epsilon)(1+\epsilon^{2})|B_{S}(N)|}.
 \end{align*}	
 	Take ${\bf j}=(j_{1},\cdots, j_{k})\in [1, \ell_{i_{1}}]\times \cdots \times [1, \ell_{i_{k}}]$ with $\mathcal{X}_{E, {\bf j}}\neq \emptyset$. Since
 	\begin{align*}
 	\sum \limits_{F_{i,j}a\in \mathcal{F}}|F_{i, j}a|\leq \dfrac{1}{1-10\delta^{1/4}}|B_{S}(N)|<(1+\epsilon^{2})|B_{S}(N)|~\text{by} ~(\ref{da}),
 	\end{align*}
 we have
 \begin{align*}
 2^{-\frac{1}{c_{2}}(s-\epsilon)(1+\epsilon^{2})|B_{S}(N)|}\leq \prod_{m=1}^{k}2^{-\frac{1}{c_{2}}(s-\epsilon)|F_{i_{m}, j_{m}}|}.
 \end{align*}	
 	Moreover
 	\begin{align*}
 |\pi_{B_{S}(N)}(\mathcal{X}_{E})|\cdot 2^{-\frac{1}{c_{2}}(s-\epsilon)(1+\epsilon^{2})|B_{S}(N)|}\leq \sum\limits_{{\bf j}}|A|^{(\delta+\delta^{1/4})|B_{S}(N)|}	\prod_{m=1}^{k}2^{-\frac{1}{c_{2}}(s-\epsilon)|F_{i_{m}, j_{m}}|}.	
 	\end{align*}
 	The right-hand side isn't more than
 	\begin{align*}
 	|A|^{(\delta+\delta^{1/4})|B_{S}(N)|}\left( \sum\limits_{j=1}^{l_{i_{1}}}2^{-\frac{1}{c_{2}}(s-\epsilon)|F_{i_{1},j}|}\right) \times \cdots \times \left( \sum\limits_{j=1}^{l_{i_{k}}}2^{-\frac{1}{c_{2}}(s-\epsilon)|F_{i_{k},j}|}\right). 
 	\end{align*}
 According to (\ref{e1}), we have
 \begin{align*}
 |\pi_{B_{S}(N)}(\mathcal{X}_{E})|\cdot 2^{-\frac{1}{c_{2}}(s-\epsilon)(1+\epsilon^{2})|B_{S}(N)|}\leq |A|^{(\delta+\delta^{1/4})|B_{S}(N)|}.
 \end{align*}	
 \end{proof}
We can now proceed to prove Theorem \ref{main}.
 \begin{align*}
 |\pi_{B_{S}(N)}(\mathcal{X}_{E})|\cdot 2^{-\frac{1}{c_{2}}(s-\epsilon)(1+\epsilon^{2})|B_{S}(N)|}&\leq |A|^{(\delta+\delta^{1/4})|B_{S}(N)|}\\&\leq 2^{|B_{S}(N)|\epsilon^{3}/4c_{2}}. ~\text{by}~(\ref{da})
 \end{align*}
If $N$ is sufficiently large, the number of choices $E\subset B_{S}(N)\times [1,M]$ such that $\mathcal{X}_{E}\neq \emptyset$ is not greater  than  $2^{(\epsilon^{3}/4c_{2})|B_{S}(N)|}$. Hence
 \begin{align*}
  |\pi_{B_{S}(N)}(\mathcal{X})|\cdot 2^{-\frac{1}{c_{2}}(s-\epsilon)(1+\epsilon^{2})|B_{S}(N)|}&=\sum \limits_{E~ \text{with}~\mathcal{X}_{E}\neq \emptyset}  |\pi_{B_{S}(N)}(\mathcal{X}_{E})|\cdot 2^{-\frac{1}{c_{2}}(s-\epsilon)(1+\epsilon^{2})|B_{S}(N)|}\\&< 2^{(\epsilon^{3}/4c_{2})|B_{S}(N)|}\times 2^{(\epsilon^{3}/4c_{2})|B_{S}(N)|}\\&=2^{(\epsilon^{3}/2c_{2})|B_{S}(N)|}.
 \end{align*}
 Then
 \begin{align*}
 \dfrac{\log|\pi_{B_{S}(N)}(\mathcal{X})|}{|B_{S}(N)|}<\dfrac{1}{c_{2}}(s-\epsilon+s\epsilon^{2}-\frac{1}{2}\epsilon^{3}).
 \end{align*}
 Letting $N\rightarrow \infty$, we have
 \begin{align*}
 h_{top}(\mathcal{X}, G)\leq \dfrac{1}{c_{2}}(s-\epsilon+s\epsilon^{2}-\frac{1}{2}\epsilon^{3})<\dfrac{1}{c_{2}}s=h_{top}(\mathcal{X}, G)~(\text{by} ~s\epsilon-\frac{\epsilon^{2}}{2}<1).
 \end{align*}
 This is a contradiction.

\end{proof}
\section{Proof of Theorem \ref{11}}\label{s3}
\subsection{Mutual information}
Here we prepare some basics of mutual information. Let $(\Omega, \mathbb{P})$ be a probability space. Let $\mathcal{X}$ and $\mathcal{Y}$ be measurable spaces, and let $X: \Omega \rightarrow \mathcal{X}$ and $Y: \Omega \rightarrow \mathcal{Y}$ be measurable maps. We want to define their \emph{ mutual information} $I(X;Y)$ as the measure of the amount of information $X$ and $Y$ share. For  more details and properties of mutual information, one is referred  to \cite{CT06}.

{\bf Case 1:} Suppose $\mathcal{X}$ and $\mathcal{Y}$ are finite sets.  Then we define 
\begin{align*}
I(X; Y)= H(X)+ H(Y)-H(X,Y)=H(X)- H(X| Y).
\end{align*} 
More explicitly 
$$ I(X; Y)=\sum\limits_{x\in X, y\in Y} \mathbb{P}(X=x, Y=y)\log\dfrac{\mathbb{P}(X=x, Y=y)}{\mathbb{P}(X=x)\mathbb{P}(Y=y)}.$$
Here we use the convention that $0\log (0/a)=0$ for all $a\leq 0$.

{\bf Case 2:} In general,  take measurable maps $f: \mathcal{X} \rightarrow A$ and $g: \mathcal{Y} \rightarrow B$ into finite sets $A$ and $B$. Then we can consider $I(f\circ  X; g\circ Y)$ defined  by Case 1. We define $ I(X; Y)$  as the supremum of  $I(f\circ X; g\circ Y)$ over all finite-range measurable maps $f$ and $g$ defined on $\mathcal{X}$ and $\mathcal{Y}$. This definition is compatible with Case 1 when $\mathcal{X}$ and $\mathcal{Y}$ are finite sets.

\begin{lem}[Date-Processing inequality]\label{dp}
	Let $X$ and $Y$ be random variables taking values in measurable spaces $ \mathcal{X}$ and $\mathcal{Y}$ respectively. If $f:\mathcal{Y} \rightarrow \mathcal{Z}  $ is a measurable map then $I(X; f(Y)) \leq  I(X; Y)$. 	
\end{lem} 
\subsection{Rate distortion theory}
 We introduce   rate distortion function and dimension. Let $(\mathcal{X},G)$  be a dynamical system with a distance $d$ on $\mathcal{X}$. Take an invariant probability $\mu\in M(\mathcal{X},G)$. For  a positive number $\epsilon$ and $F\in F(G)$, we define  $R_{\mu}(\epsilon, F)$ as the infimum of 
\begin{align}\label{mu}
{I(X,Y)},
\end{align}
$X$ and $Y=(Y_{0}, \cdots, Y_{n-1})$ are random variables defined on  some probability space $(\Omega, \mathbb{P})$ such that 
\begin{itemize}
	\item $X$ takes values in $\mathcal{X}$  and its law is given by $\mu$.
	\item Each $Y_{g}$ takes values in $\mathcal{X}$ and $Y=(Y_{g})_{g \in F}$ approximates the process $(gX)_{g \in F}$ in the sense that 
	\begin{align}\label{co1}
	\mathbb{E}\left( \dfrac{1}{|F|}\sum\limits_{g \in F}d(gX, Y_{g})\right) < \epsilon.
	\end{align}
	
\end{itemize}
Here $\mathbb{E}$ is the expectation with respect to the probability measure $\mathbb{P}$. 
Note that $R_{\mu}(\epsilon, F)$ depends on the distance $d$ although it is not explicitly  written in the notation.

We define the rate distortion function $$R_{\mu}(\left\lbrace B_{S}(N) \right\rbrace,  \epsilon)=\limsup\limits_{N\rightarrow \infty } \dfrac{R_{\mu}(\epsilon, B_{S}(N))}{|B_{S}(N)|}.$$ 
The upper and lower rate distortion dimensions are defined by 
\begin{align*}
\overline{\rm rdim}(\mathcal{X}, \left\lbrace B_{S}(N)\right\rbrace , d, \mu)=\limsup\limits_{\epsilon \rightarrow 0}\dfrac{R_{\mu}(\left\lbrace B_{S}(N) \right\rbrace,  \epsilon)}{\log (1/\epsilon)},
\end{align*}
\begin{align*}
\underline{\rm rdim}(\mathcal{X}, \left\lbrace B_{S}(N)\right\rbrace , d, \mu)=\liminf\limits_{\epsilon \rightarrow 0}\dfrac{R_{\mu}(\left\lbrace B_{S}(N) \right\rbrace,  \epsilon)}{\log (1/\epsilon)}.
\end{align*}
\subsection{Proof of Theorem \ref{11}}
	The proof of Theorem \ref{11} is  divided into two steps.
	
{\textbf{Step 1:}}  $\overline{\rm rdim}(\mathcal{X}, \sigma_{1}, \left\lbrace B_{S_{1}}(N) \right\rbrace, d, \mu)\leq c\cdot h_{\mu}(\mathcal{X}, G)$.
\begin{proof}	
 Let $X$ be a random variable taking values in $\mathcal{X}$ and obeying $\mu$. Let $0<\epsilon <1$ and choose $M>0$ such that $2^{-M}< \epsilon \leq 2^{-M+1}$. Given $N>0$ and  for every point $x\in \pi_{B_{S_{1}}(M)B_{S_{1}}(N)\times B_{S_{2}}(M)}(\mathcal{X})$ we take $q(x)\in \mathcal{X}$ satisfying	$\pi_{B_{S_{1}}(M)B_{S_{1}}(N)\times B_{S_{2}}(M)}(q(x))=x$. Let $X'=q(\pi_{B_{S_{1}}(M)B_{S_{1}}(N)\times B_{S_{2}}(M)}(\mathcal{X}))$ and $Y=(\sigma_{1, g}X')_{g\in B_{S_{1}}(N)}$, we conclude that 
\begin{align*}
&\dfrac{1}{|B_{S_{1}}(N)|}\sum \limits_{g\in B_{S_{1}}(N)} d(\sigma_{1, g} X, Y_{g})=\dfrac{1}{|B_{S_{1}}(N)|}\sum \limits_{g\in B_{S_{1}}(N)} d(\sigma_{1, g} X, \sigma_{1, g} X') \leq 2^{-M}<\epsilon, 
\end{align*}
and 
\begin{align*}
I(X;Y)\leq H(Y)=H(X')=H\left( (X_{g'})_{g'\in {B_{S_{1}}(M)B_{S_{1}}(N)\times B_{S_{2}}(M)}} \right). 
\end{align*}	
This yields that 
\begin{align*}
R_{\mu}(\left\lbrace B_{S_{1}}(N) \right\rbrace, \epsilon) \leq  \dfrac{I(X;Y)}{|B_{S_{1}}(N)|}\leq \dfrac{H\left( (X_{g'})_{g'\in {B_{S_{1}}(M)B_{S_{1}}(N)\times B_{S_{2}}(M)}} \right) }{|B_{S_{1}}(N)|},
\end{align*}	
and 
\begin{align*}
\dfrac{R_{\mu}(\left\lbrace B_{S_{1}}(N) \right\rbrace, \epsilon)}{\log(1/\epsilon)}\leq \dfrac{H\left( (X_{g'})_{g'\in \pi_{B_{S_{1}}(M)B_{S_{1}}(N)\times B_{S_{2}}(M)}} \right) }{|B_{S_{1}}(M)B_{S_{1}}(N)\times B_{S_{2}}(M)|} \times \dfrac{|B_{S_{1}}(M)B_{S_{1}}(N)\times B_{S_{2}}(M)|}{|B_{S_{1}}(N)|(M-1)}.
\end{align*}	
	Letting  $N\rightarrow \infty$ and then take  $\epsilon \rightarrow 0$. Note that
$\left\lbrace B_{S_{1}}(N) \right\rbrace $ is a F$\phi$lner sequence,
we get $$\lim\limits_{N\rightarrow \infty} \dfrac{|B_{S_{1}}(M)B_{S_{1}}(N)|}{|B_{S_{1}}(N)|}=1.$$ Since 
 $\lim\limits_{M\rightarrow \infty} \dfrac{|B_{S_{2}}(M)|}{M-1}=c$, 	
then
\begin{align*}
	\overline{\rm rdim}(\mathcal{X}, \sigma_{1}, \left\lbrace B_{S_{1}}(N) \right\rbrace, d, \mu)\leq c\cdot h_{\mu}(\mathcal{X}, G).
\end{align*}	
\end{proof}

For the proof of Theorem \ref{11}, we need the following lemma. The proof of this idea is adapted from  \cite{LT18} \cite{Dou}.
\begin{lem}\label{key}
Let $N\geq 1$ and $B$ a finite set. Let $X=(X_{g})_{g\in B_{S_{1}}(N)}$  and $Y=(Y_{g})_{g\in B_{S_{1}}(N)}$ be random variables taking values in $B^{B_{S_{1}}(N)}$ (namely, each $X_{g}$ and $Y_{g}$ takes values in $B$) such that for some $0<\delta < 1/2$
\begin{align*}
\mathbb{E}(\# \left\lbrace g \in B_{S_{1}}(N): X_{g}\neq  Y_{g} \right\rbrace )< \delta |B_{S_{1}}(N)|.
\end{align*}
Then 
\begin{align*}
I(X ; Y)> H(X)- |B_{S_{1}}(N)|H(\delta)-\delta |B_{S_{1}}(N)| \log |B|.
\end{align*}
\end{lem}
\begin{proof}
	Let $Z_{g}=1_{\left\lbrace X_{g}\neq Y_{g}\right\rbrace }$ and $Z=\left\lbrace g\in B_{S_{1}}(N) | X_{g}\neq Y_{g}\right\rbrace $. We can identity $Z$ with $(Z_{g})_{g\in B_{S_{1}}(N)}$ and hence
	\begin{align*}
	H(Z)&\leq \sum\limits_{g\in B_{S_{1}}(N)} H(Z_{g})=\sum\limits_{g\in B_{S_{1}}(N)} H(\mathbb{E} Z_{g}) \\&\leq |B_{S_{1}}(N)| H\left( \dfrac{1}{B_{S_{1}}(N)}\sum\limits_{g\in B_{S_{1}}(N)}\mathbb{E} Z_{g}\right) <|B_{S_{1}}(N)|H(\delta).
	\end{align*}
	Then  $H(Z)<|B_{S_{1}}(N)|H(\delta)$. We expand $H(X, Z|Y)$ in two ways:
	\begin{align*}
	H(X, Z|Y)=H(X|Y)+H(Z|X, Y)=H(Z|Y)+H(X|Y, Z).
	\end{align*}
	$H(Z|X, Y)=0$ because $Z$ is determined by $X$ and $Y$. Form this, we conclude that
	\begin{align*}
	H(X|Y)=H(Z|Y)+H(Z|Y,Z)<|B_{S_{1}}(N)|H(\delta)+H(X|Y,Z).
	\end{align*}
	Noticing that
	$$H(X|Y, Z)=\sum\limits_{E\subset B_{S_{1}}(N)} \mathbb{P}(Z=E)H(X|Y, Z=E).$$
	For $Y$ and the condition $Z=E$, the possibilities of $X$ is at most $|B|^{|E|}$. So
	$H(X|Y, Z=E)\leq |E|\log |B|$ and 
	\begin{align*}
	H(X|Y, Z)&\leq \sum\limits_{E\subset B_{S_{1}}(N)} |E|\mathbb{P}(Z=E) \log |B|\\&=\mathbb{E}|Z|\cdot \log |B|\\&\leq \delta |B_{S_{1}}(N)| \log |B|.
	\end{align*}
	Thus $H(X|Y)<|B_{S_{1}}(N)|H(\delta)+ \delta |B_{S_{1}}(N)| \log |B|$ and $$I(X ; Y)> H(X)- |B_{S_{1}}(N)|H(\delta)-\delta |B_{S_{1}}(N)| \log |B|.$$
	
\end{proof}

\textbf{Step 2:} $\underline{\rm rdim}(\mathcal{X}, \sigma_{1}, d, \left\lbrace B_{S_{1}}(N)\right\rbrace,  \mu)\geq c\cdot h_{\mu}(\mathcal{X}, G)$.
\begin{proof}  
	Let $X$ be a random variable taking values in $\mathcal{X}$ with ${\rm Law}(X)=\mu$. Given $0<\epsilon<\delta<1/2$, $N>0$ and  let $Y=(Y_{g})_{g\in B_{S_{1}}(N)}$ be a random variable taking values in $\mathcal{X}^{B_{S_{1}}(N)}$ and satisfying 
	\begin{align*}
	\mathbb{E}\left( \dfrac{1}{|B_{S_{1}}(N)|} \sum \limits_{g\in B_{S_{1}}(N)} d(\sigma_{1, g}X, Y_{g})\right) <\epsilon.
	\end{align*}
	We will estimate the lower bound of  $I(X;Y)$ . Choose $M\geq 0$ with $\delta 2^{-M-1} < \epsilon < \delta 2^{-M}$. For $g\in B_{S_{1}}(N)$,  set
	\begin{align*}
	&X_{g}'=\pi_{\left\lbrace g\right\rbrace \times B_{S_{2}}(M)}(X)=(X_{(g, g_{2})})_{g_{2}\in B_{S_{2}}(M)}, \\& Y_{g}'=\pi_{\left\lbrace 1_{G_{1}}\right\rbrace \times B_{S_{2}}(M)}(Y_{g})=((Y_{g})_{(1_{G_{1}}, g_{2})})_{g_{2}\in B_{S_{2}}(M)}.
	\end{align*}
	If $X_{g}'\neq Y_{g}'$ for some $g$ then $d(\sigma_{1, g}X, Y_{g})\geq 2^{-M}$. Therefore  
	$\mathbb{E}d(\sigma_{1, g} X, Y_{g})\geq 2^{-M} \mathbb{P}(X_{g}'\neq Y_{g}')$ and 
	\begin{align*}
	\mathbb{E}( \#\left\lbrace  g \in B_{S_{1}}(N): X_{g}'\neq Y_{g}'\right\rbrace )&=\sum\limits_{g\in B_{S_{1}}(N)} \mathbb{P}(X_{g}'\neq Y_{g}')\\&\leq 2^{M}\mathbb{E}\left( \sum\limits_{g\in B_{S_{1}}(N)}d(\sigma_{1, g} X, Y_{g})\right) \\&< 2^{M}\epsilon |B_{S_{1}}(N)|\leq \delta |B_{S_{1}}(N)|.
	\end{align*}
Applying Lemma \ref{key} to $X_{g}'$ and $Y_{g}'$ with $B=A^{B_{S_{2}}(M)}$:
\begin{align*}
&I\left( (X_{g}')_{g\in B_{S_{1}}(N)}; (Y_{g}')_{g\in B_{S_{1}}(N)}\right)\\&> H((X_{g}')_{g\in B_{S_{1}}(N)})-|B_{S_{1}}(N)|H(\delta)-\delta |B_{S_{1}}(N)||B_{S_{2}}(M)|\log |A|.
\end{align*}	
 According to  the data-processing inequality (Lemma  \ref{dp}),
 \begin{align*}
 I(X; Y)\geq I((X_{g}')_{g\in B_{S_{1}}(N)}; (Y_{g}')_{g\in B_{S_{1}}(N)}).
 \end{align*}
	Then 
	\begin{align*}
	\dfrac{I(X; Y)}{|B_{S_{1}}(N)|}\geq \dfrac{H\left\lbrace (X_{g'})_{g' \in B_{S_{1}}(N)\times B_{S_{2}}(M)} \right\rbrace }{|B_{S_{1}}(N)|}-H(\delta)-\delta |B_{S_{2}}(M)|\log |A|.
	\end{align*}
This holds for any $N>0$. So
\begin{align*}
R_{\mu}(\left\lbrace B_{S_{1}}(N) \right\rbrace,  \epsilon)&\geq  \inf_{N} \dfrac{H\left\lbrace (X_{g'})_{g' \in B_{S_{1}}(N)\times B_{S_{2}}(M)} \right\rbrace }{|B_{S_{1}}(N)|}- H(\delta)-\delta |B_{S_{2}}(M)|\log |A|\\&= \lim\limits_{N\rightarrow \infty}\dfrac{H\left\lbrace (X_{g'})_{g' \in B_{S_{1}}(N)\times B_{S_{2}}(M)} \right\rbrace }{|B_{S_{1}}(N)|}- H(\delta)-\delta |B_{S_{2}}(M)|\log |A| .
\end{align*}	
We divide this by $\log (1/\epsilon)$ and take the limit $\epsilon\rightarrow 0$. 
Since $\log (1/\epsilon)< \log (1/ \delta)+ (M+1)$ (here $\delta$ has been fixed) and  $\lim\limits_{M\rightarrow \infty} \dfrac{|B_{S_{2}}(M)|}{M-1}=c$, we obtain
\begin{align*}
{\overline{\rm rdim}}(\mathcal{X}, \sigma_{{1}}, \left\lbrace B_{S_{1}}(N) \right\rbrace, d, \mu)\geq c \cdot h_{\mu}(\mathcal{X}, G)-c\delta  \log|A|.
\end{align*}
Here we have used 
\begin{align*}
h_{\mu}(\mathcal{X}, G)=\lim\limits_{N, M\rightarrow \infty} \dfrac{H\left\lbrace \left( X_{g'}\right)_{g'\in B_{S_{1}}(N)\times B_{S_{2}}(M)} \right\rbrace }{|B_{S_{1}}(N)||B_{S_{2}}(M)|}
\end{align*}
Letting $\delta\rightarrow 0$, we have $\underline{\rm rdim}(\mathcal{X}, \sigma_{1}, \left\lbrace B_{S_{1}}(N) \right\rbrace, d, \mu) \geq  c h_{\mu}(\mathcal{X}, G)$.
\end{proof}
\section{Examples}\label{s4}
In this section, we consider the following examples to illustrate our main theorem for the case of $c_{1}=c_{2}=c$. (see \cite{CC} for more details )
\begin{example}
 Let $G_{1}=\mathbb{Z}^{d}$, $S_{1}=\left\lbrace (1, 0, \cdots, 0),\cdots, (0,  \cdots, 1), (0, 0, \cdots, -1) \right\rbrace $.	Let $G_{2}=\mathbb{Z}\times (\mathbb{Z}/2\mathbb{Z})$ and $S_{2}=\left\lbrace (1, \overline{0}), (-1, \overline{0}), (0, \overline{1})\right\rbrace $. Then the ball radius $r$ centered at the element $(n, \overline{0})$ is represented in Fig\ref{fig0}. We deduce that $\gamma_{S_{2}}(n)=(2n+1)+2(n-1)+1=4n$. Take $G=G_{1}\times G_{2}$, by Theorem \ref{main}, we can have $${\rm mdim}_{H}(\mathcal{X}, \left\lbrace B_{S_{1}}(n) \right\rbrace , d)={\rm mdim}_{M}(\mathcal{X}, {G_{1}}, d)= 4  h_{top}(\mathcal{X}, {G}).$$ 
\end{example}

\begin{figure}[H]
	\centering
	\includegraphics[width=9.5cm, height=2.2cm]{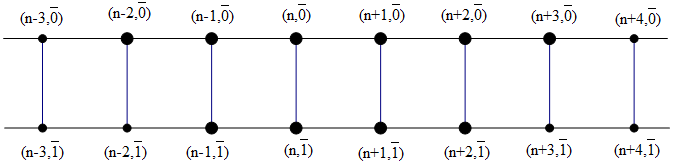}
	\caption{}\label{fig0}
\end{figure}

Recall infinite dihedral group, that is,  the group  of  isometries of the real line $\mathbb{R}$ generated by reflections $r: \mathbb{R} \rightarrow \mathbb{R}$ and $s: \mathbb{R} \rightarrow \mathbb{R}$ defined by 
\begin{align*}
&r(x)=-x ~~(\text{symmetry with respect to 0})\\&
s(x)=1-x~~(\text{symmetry with respect to 1/2})
\end{align*}
for all $x\in \mathbb{R}$. Note that $r^{2}=s^{2}=1_{G}$. 
\begin{example}
Let $G_{1}=\mathbb{Z}^{d}$, $S_{1}=\left\lbrace (1, 0, \cdots, 0),\cdots, (0,  \cdots, 1), (0, 0, \cdots, -1) \right\rbrace $.	Let $G_{2}$ be the infinite dihedral group and $S_{2}=\left\lbrace r, s \right\rbrace $. Then the ball of radius $n$ centered at the element $g\in G_{2}$ is represented in Fig \ref{fig1}.  It follows that $\gamma_{S_{2}}(n)=2n+1$.  Take  $G=G_{1}\times G_{2}$. By Theorem \ref{main}, thus 
	${\rm mdim}_{H}(\mathcal{X}, \left\lbrace B_{S_{1}}(n) \right\rbrace , d)={\rm mdim}_{M}(\mathcal{X}, {G_{1}}, d)= 2\cdot h_{top}(\mathcal{X}, {G}).$
\end{example}

\begin{figure}[H]
	\centering
	\includegraphics[width=13.6cm, height=1.0cm]{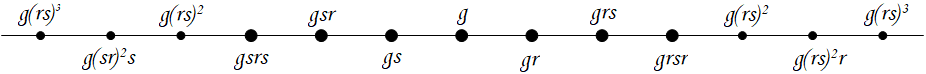}
	\caption{}\label{fig1}
\end{figure}

  Let $G$ be a group. The \emph{lower central series of } $G$ is the sequence $(C^{i}(G))_{i\geq 0}$ of subgroup of $G$ defined by $C^{0}(G)=G$ and $C^{i+1}=[C^{i}(G), G]$ for all $i\geq 0$.   Here $[h, k]:=hkh^{-1}k^{-1}$ for $h, k\in  G$. An easy induction shows that $C^{i}(G)$ is normal in $G$ and that $G^{i+1}(G)\subset C^{i}(G)$ for all $i$. The group $G$ is said to be \emph{nilpotent} if there is an integer $i\geq 0$ such that $C^{i}(G)=\left\lbrace 1_{G} \right\rbrace $. The group $G$ is said to be nilpotent if there is an integer $i\geq 0$ such that $C^{i}(G)=\left\lbrace 1_{G}\right\rbrace $. The smallest integer $i\geq 0$ such that  $C^{i}(G)=\left\lbrace 1_{G} \right\rbrace $ is then called the nilpotency  degree of $G$. Every  nilpotent group is amenable. 
  
  1972, H. Bass \cite{Bas} showed that the growth of a nilpotent group $G$ with finite symmetric generating subset $S$ is exactly polynomial in the sense that there are positive constants $C_{1}$ and $C_{2}$ such that $C_{1}n^{d}\leq \gamma_{S}(n) \leq C_{2}n^{d}$, for all $n\geq 1$, where $d=d(G)$ is an integer which can be computed explicitly from the lower central series of $G$.
  \begin{example}
Let $G_{1}$ be a nilpotent group with finite symmetric generating subset $S_{1}$ and $deg(G_{1})=d$. Let $G_{2}$ be the infinite dihedral group   and $S_{2}=\left\lbrace  r, s\right\rbrace $. Set $G=G_{1}\times G_{2}$. Then 
$${\rm mdim}_{H}(\mathcal{X}, \left\lbrace B_{S_{1}}(n) \right\rbrace , d)={\rm mdim}_{M}(\mathcal{X}, {G_{1}}, d)= 2\cdot h_{top}(\mathcal{X}, {G}).$$	
\end{example}
Similarly, the above examples hold for Theorem \ref{11}.

{\bf Acknowledgements.}
The first author was supported by
 the Postgraduate Research Innovation Program of Jiangsu Province (KYCX201162). The first and second author were supported by NNSF of China (11671208 and 11431012). The third author was supported by NNSF of China (11971236,11601235), NSF of Jiangsu Province (BK20161014), NSF of the Jiangsu Higher Education Institutions of China (16KJD110003), China Postdoctoral Science Foundation (2016M591873), and China Postdoctoral Science Special Foundation (2017T100384). The work was also funded by the Priority Academic Program Development of Jiangsu Higher Education Institutions. We would like to express our gratitude to Tianyuan Mathematical Center in Southwest China, Sichuan University and Southwest Jiaotong University for their support and hospitality.

\end{document}